\documentclass[graybox,envcountsec]{svmult}

\usepackage{mathptmx}       
\usepackage{helvet}         
\usepackage{courier}        
\usepackage{type1cm}        
\usepackage{makeidx}         
\usepackage{graphicx}        
\usepackage{multicol}        
\usepackage[bottom]{footmisc}

\usepackage{amsmath}
\usepackage{amssymb}
\usepackage{amscd}
\usepackage{indentfirst}

\usepackage{mathptmx}
\DeclareMathAlphabet{\mathcal}{OMS}{cmsy}{m}{n}
\usepackage{mathtools} 

\usepackage{xcolor}

\usepackage{enumitem}

\usepackage{changes}
\definechangesauthor[name={Ha}, color=red]{Ha}
\definechangesauthor[name={Ar}, color= green]{Ar}
\definechangesauthor[name={Va}, color=blue]{Va}
\setremarkmarkup{(#2)}

\newtheorem{thm}{Theorem}[section]
\newtheorem{cor}[thm]{Corollary}
\newtheorem{lem}[thm]{Lemma}

\newtheorem{assump}[thm]{Assumption}

\newtheorem{rem}[thm]{Remark}

\newcommand{\bt}{\begin{thm}}
\newcommand{\et}{\end{thm}}
\newcommand{\bl}{\begin{lem}}
\newcommand{\el}{\end{lem}}
\newcommand{\bc}{\begin{cor}}
\newcommand{\ec}{\end{cor}}
\newcommand{\br}{\begin{rem}}
\newcommand{\er}{\end{rem}}

\newcommand{\ba}{\begin{array}}
\newcommand{\ea}{\end{array}}
\newcommand{\bea}{\begin{eqnarray}}
\newcommand{\eea}{\end{eqnarray}}
\newcommand{\bead}{\begin{eqnarray*}}
\newcommand{\eead}{\end{eqnarray*}}
\newcommand{\be}{\begin{equation}}
\newcommand{\ee}{\end{equation}}
\newcommand{\bed}{\begin{displaymath}}
\newcommand{\eed}{\end{displaymath}}
\newcommand{\bps}{\begin{split}}
\newcommand{\eps}{\end{split}}
\newcommand{\la}{\label}

\newcommand\dR{{\mathbb{R}}}
\newcommand\dC{{\mathbb{C}}}

\newcommand\dN{{\mathbb{N}}}

\newcommand\gotH{{\mathfrak{H}}}

\newcommand\gotK{{\mathfrak{K}}}

\newcommand{\ga}{{\alpha}}
\newcommand{\gb}{{\beta}}
\newcommand{\gd}{{\delta}}
\newcommand{\gD}{{\Delta}}
\newcommand{\gga}{{\gamma}}

\newcommand{\gL}{{\Lambda}}

\newcommand{\gs}{{\sigma}}

\newcommand\gt{{\tau}}

\newcommand{\gY}{\Upsilon}

\newcommand\cA{{\mathcal{A}}}
\newcommand\cB{{\mathcal{B}}}

\newcommand\cD{{\mathcal{D}}}

\newcommand\cI{{\mathcal{I}}}

\newcommand\cK{{\mathcal{K}}}
\newcommand\cL{{\mathcal{L}}}

\newcommand\cT{{\mathcal{T}}}

\newcommand\cU{{\mathcal{U}}}

\newcommand{\dom}{\mathrm{dom}}

\newcommand{\cl}{\mathcal}

\DeclareMathOperator*{\esssup}{ess\,sup}

\def\1{{\bf 1}}

\def\wt#1{{{\widetilde #1} }}

\newcommand{\slim}{\,\mbox{\rm s-}\hspace{-2pt} \lim}

\newcommand{\eproof}{\hfill$\Box$}

\makeindex             

\begin{document}

\title*{Trotter product formula and linear evolution equations on Hilbert spaces\\
\vspace{5mm}
\normalsize \it On the occasion of the 100th birthday of Tosio Kato}
\titlerunning{Trotter product formula and evolution semigroups}

\author{Hagen Neidhardt, Artur Stephan, Valentin A. Zagrebnov}
\authorrunning{H.~Neidhardt, A.~Stephan, V.~A.~Zagrebnov}

\institute{H.~Neidhardt \at Weierstrass Institute for Applied Analysis and Stochastics,\\
Mohrenstr. 39, D-10117 Berlin, Germany\\
\email{hagen.neidhardt@wias-berlin.de}
\and
A.~Stephan \at
Institut f\"ur Mathematik,
Humboldt-Universit\"at zu Berlin,\\
Unter den Linden 6,
D-10099 Berlin, Germany\\
\email{stephan@math.hu-berlin.de}
\and
V.~A.~Zagrebnov \at
Institut de Math\'{e}matiques de Marseille  (UMR 7373)\\
Universit\'{e} d'Aix-Marseille,
CMI - Technop\^{o}le Ch\^{a}teau-Gombert\\
39 rue F. Joliot Curie, 13453 Marseille, France\\
\email{valentin.zagrebnov@univ-amu.fr}}

\maketitle

\abstract{The paper is devoted to evolution equations of the form
\bed
\frac{\partial}{\partial t}u(t) = -(A + B(t))u(t), \quad t \in \cI = [0,T],
\eed
on separable Hilbert spaces where $A$ is a non-negative self-adjoint operator and $B(\cdot)$ is family of non-negative self-adjoint operators
such that $\dom(A^\ga) \subseteq \dom(B(t))$ for some $\ga \in [0,1)$ and the map $A^{-\ga}B(\cdot)A^{-\ga}$ is  H\"older continuous
with the H\"older exponent $\gb \in (0,1)$. It is shown that the solution operator $U(t,s)$ of the evolution equation can be approximated in the operator norm
by a combination of semigroups generated by $A$ and $B(t)$  provided the condition $\gb > 2\ga -1$ is satisfied. The convergence rate for the approximation is given by the H\"older exponent $\gb$.
The result is proved using the evolution semigroup approach.
}

\numberwithin{equation}{section}
\renewcommand{\theequation}{\arabic{section}.\arabic{equation}}

\section{Introduction} \label{sec:1}

\noindent
A closer look to Kato's work shows that abstract evolution equations and Trotter product formula were topics of high interest for Kato.
Already at the beginning of his scientific career Kato was interested in evolution equations \cite{Kato1953,Kato1956}. This interest has lasted a lifetime
\cite{Kato1961,Kato1965,Kato1970,Kato1973,Kato2011a,Kato2011b,KatoTan1962}. Another topic of great interest for him was the so-called Trotter product formula
\cite{Kato1978,Kato1978b,Kato1974,Kato1978c}. Even the paper \cite{Kato1978} has inspired further developments in this field \cite{ITTZ2001}.

The topic of the present paper is to link evolution equations with the Trotter product formula. To this end
we consider an abstract evolution equation of type
\be\la{eq:1.1}
\begin{split}
\frac{\partial u(t)}{\partial t} = & -C(t) u(t), \quad u(s) = x_s, \quad s \in [0,T),\\
C(t) = & A +B(t),
\end{split}
\qquad t \in \cI := [0,T],
\ee
on the separable Hilbert space $\gotH$. Evolution equations of that type on Hilbert or Banach spaces are widely investigated, cf. \cite{AcquistapaceTerreni1984,AcquistapaceTerreni1985,
Amann1988,Amann1987,ArendtDierOuhabaz2014,KatoTan1962,Lunardi1987,MonniauxRhandi2000,Tan1959,Tan1960a,Tan1960b,
Tan1961,Tan1967,Tan1967b,Yagi1976,Yagi1977,Yagi1988,Yagi1989,Yagi1990} or the books
\cite{Amann1995,Tan1979,Yagi2010}.
We consider the equation \eqref{eq:1.1} under the following assumptions.

\begin{assump}
{\em
\begin{enumerate}[label=(S\arabic*), leftmargin=3\parindent]
\item[]

\item The operator $A$ is self-adjoint in the Hilbert space $\gotH$ such that $A \ge I$.
Let $\{B(t)\}_{t \in \cI}$ be a family of non-negative self-adjoint operators in $\gotH$ such that
the function $(I + B(\cdot))^{-1}: \cI \longrightarrow \cL(\gotH)$ is strongly measurable.

\item There is an $\ga \in [0,1)$ such that for a.e. $t \in \cI$ the inclusion $\dom(A^\ga) \subseteq \dom(B(t))$
holds. Moreover, the function $B(\cdot)A^{-\ga}: \cI \longrightarrow \cL(\gotH)$
is strongly measurable and essentially bounded, i.e.
\be
C_\ga := \esssup_{t\in \cI}\|B(\cdot)A^{-\ga}\| < \infty.
\ee

\item

The map $A^{-\ga}B(\cdot)A^{-\ga}: \cI \longrightarrow \cL(\gotH)$ is H\"older continuous, i.e, for some $\gb \in (0,1)$ there is a constant $L_{\ga,\gb} > 0$  such that
the estimate
\be\la{eq:2.2}
\|A^{-\ga}(B(t) - B(s))A^{-\ga}\| \le L_{\ga,\gb} |t-s|^\gb, \quad (t,s) \in \cI \times \cI,
\ee
holds.\hfill$\triangle$
\end{enumerate}
}
\end{assump}

Notice that under the assumption (S2) the operator $C(t)$
is also an invertible non-negative self-adjoint operator for each $t \in \cI$.
Assumptions of that type were made in \cite{FujTan1973,IchinoseTamura1998,NeiStephZagr2016,NeiStephZagr2017,Yagi1990}.
One checks that the assumptions (S1)-(S3) and the additional assumption $\gb > \ga$ imply the assumptions (I), (VI) and (VII) of \cite{Yagi1990} for the family $\{C(t)\}_{\in\cI}$.
Hence, Proposition 3.1 and Theorem 3.2 of \cite{Yagi1990} yield the existence of a so-called
\textit{solution} (or \textit{evolution}) operator
for the evolution equation \eqref{eq:1.1}, i.e., a strongly continuous, uniformly bounded family of
bounded operators
$\{U(t,s)\}_{(t,s) \in \gD}$, $\gD~:=~\{(t,s) \in \cI \times \cI: 0 \le s \le t\le T\}$,
such that the conditions
\be\la{eq:1.2}
\begin{split}
U(t,t) =& I, \quad \mbox{for} \quad t \in \cI,\\
U(t,r)U(r,s) =& U(t,s), \quad \mbox{for} \quad t,r,s \in \cI \quad \mbox{with} \quad s \le r \le t,
\end{split}
\ee
are satisfied and $u(t) = U(t,0)x$ is for every $x \in \gotH$ a strict solution of \eqref{eq:1.1}, see Definition 1.1 of \cite{Yagi1990}. Because the involved operators are self-adjoint and non-negative one checks that the solution operator consists of contractions.

The aim of the present paper is to analyze  the convergence of the following approximation to the solution operator
$\{U(t,s)\}_{(t,s)\in\gD}$.
 Let
\be
s =: t_0 < t_1 < \ldots < t_{n-1} < t_n := t, \quad t_j := s + j\tfrac{t-s}{n},
\ee
$j = \{0,1,2,\ldots,n\}$, $ n \in \dN$,
be a partition of the interval $[s,t]$. Let
\be\la{eq:1.6}
\begin{split}
G_j(t,s;n) :=& e^{-\tfrac{t-s}{n} A}e^{-\tfrac{t-s}{n} B(t_j)}, \quad j = 0,1,2,\ldots,n,\\
V_n(t,s) :=& G_{n-1}(t,s;n)G_{n-1}(t,s;n)\times \cdots \times  G_2(t,s;n) G_0(t,s;n),\\
\end{split}
\ee
$n \in \dN$. The main result in the paper is the following. If the assumptions (S1)-(S3) are satisfied and in addition the condition $\gb > \ga$ holds,
then the solution operator $\{U(t,s)\}_{(t,s)\in \gD}$ of \cite{Yagi1990} admits the approximation
\be\la{eq:1.8v}
\esssup_{(t,s)\in \gD}\|V_n(t,s) - U(t,s)\| \le \frac{ R_\gb}{n^\gb}, \quad n \in \dN,
\ee
with some constant $R_\gb > 0$. The result shows that the convergence of the approximation $\{V_n(t,s)\}_{(t,s)\in \gD}$ is determined by the smoothness of the perturbation $B(\cdot)$.

If the map $A^{-\ga}B(\cdot)A^{-\ga}: \cI \longrightarrow \cL(\gotH)$ is Lipschitz continuous, then the map is of course H\"older continuous with any exponent $\gga \in (\ga,1)$.
Hence from \eqref{eq:1.8v} it immediately follows that for any $\gga \in (\ga,1)$ there is a constant $R_\gga$ such that
\be\la{eq:1.9v}
\esssup_{(t,s)\in \gD}\|V_n(t,s) - U(t,s)\| \le \frac{R_\gga}{n^\gga}, \quad n \in \dN.
\ee
In particular, for any $\gga$ close to one the estimate \eqref{eq:1.9v} holds.

In \cite{IchinoseTamura1998} the Lipschitz case was considered.
It was shown that there is a constant $\gY_0 > 0$ such that the estimate
\be\la{eq:1.9vv}
\esssup_{t \in \cI}\|V_n(t,0) - U(t,0)\| \le \gY_0\frac{\log(n)}{n}, \quad n=2,3,\ldots\,.
\ee
holds. {{It is obvious that the estimate (\ref{eq:1.9vv}) is stronger than
\bed
\esssup_{t\in\cI}\|V_n(t,0) - U(t,0)\| \le \frac{R_\gga}{n^\gga}, \quad n \in \dN.
\eed
(which follows from \eqref{eq:1.9v}) for any $\gga $ independent of how close it is to one.}

To prove \eqref{eq:1.8v} we use the so-called evolution semigroup approach which allows not only to verify the
estimate \eqref{eq:1.8v} but also to generalise it. The approach
is quite different from the technique used in \cite{IchinoseTamura1998, Yagi1990}. We have
successfully applied this approach already
in \cite{NeiStephZagr2016} and \cite{NeiStephZagr2017}.
The key idea is to forget about the evolution equation \eqref{eq:1.1}
and to consider instead of it the operators  $\cK_0$ and $\cK$ on $\gotK = L^2(\cI,\gotH)$. The operator $\cK_0$ is
the generator of the contraction semigroup $\{\cU_0(\gt)\}_{\gt\in\dR_+}$,
\be\la{eq:1.10-1v}
(\cU_0(\gt)f)(t) := e^{-\gt A}\chi_{\cI}(t-\gt)f(t-\gt), \quad f \in L^2(\cI,\gotH),
\ee
and $\cK$ is given by
\bed
\cK = \cK_0 + \cB, \quad \dom(\cK) = \dom(\cK_0) \cap \dom(\cB),
\eed
where $\cB$ is the multiplication operator induced by the family $\{B(t)\}_{\in\cI}$ in $L^2(\cI,\gotH)$ which is self-adjoint and non-negative,
for more details see Section \ref{sec:II}. It turns out that under the assumptions (S1) and (S2) the operator $\cK$ is the generator of a contraction semigroup $\{\cU(\gt)\}_{\gt\in\dR_+}$
on $L^2(\cI,\gotH)$.
For the pair $\{\cK_0,\cB\}$ we consider the Lie-Trotter product formula.  From the original paper of Trotter \cite{Trotter1959} one gets that
\be\la{eq:1.10v}
\slim_{n\to\infty}\left(e^{-\tfrac{\gt}{n}\cK_0}e^{-\tfrac{\gt}{n}\cB}\right)^n = e^{-\gt\cK}, \quad \gt \in \dR_+ := [0,\infty),
\ee
holds uniformly in $\gt$ on any bounded interval of $\dR_+$. Since $e^{-\gt\cK_0} = 0$ and $e^{-\gt\cK} = 0$ for $\gt \ge T$ one gets even uniformly in $\gt \in \dR_+$.

Previously it was shown that under certain assumptions the strong convergence can be improved to
operator-norm convergence on Hilbert spaces, see
\cite{CachNeiZag2001,CachNeiZag2002,ITTZ2001,NeidhardtZagrebnov1998,Rogava1993} as well as on Banach spaces,
see \cite{CachZag2001}. For an overview  the reader is referred to \cite{NeiStephZagr2018}. To consider the Trotter product formula for evolution equations is relatively new and was firstly realized in \cite{NeiStephZagr2016,NeiStephZagr2017}
for Banach spaces.

In the following we improve the convergence \eqref{eq:1.10v} to operator-norm convergence. We show
that under the assumptions (S1)-(S3) and $\gb > 2\ga-1$ there is a constant
$R_\gb > 0$ such that
\be\la{eq:1.11}
\sup_{\gt\in \dR_+}\left\|\left(e^{-\tfrac{\gt}{n}\cK_0}e^{-\tfrac{\gt}{n}\cB}\right)^n - e^{-\gt\cK}\right\| \le \frac{R_\gb}{n^\gb}, \quad n \in \dN,
\ee
holds.

It turns out that $\cK$ is the generators of an evolution semigroup. This means, there is a propagator $\{U(t,s)\}_{(t,s)\in\gD_0}$,
$\gD_0 := \{(t,s) \in \cI_0 \times \cI_0: s \le t\}$, $\cI_0 = (0,T]$, such that the contraction semigroup $\{\cU(\gt) = e^{-\gt \cK}\}_{\gt\in\dR_+}$ admits the representation
\be\la{eq:1.13v}
(\cU(\gt)f)(t) = U(t,t-\gt)\chi_{\cI}(t-\gt)f(t-\gt), \qquad f\in L^2(\cI,\gotH).
\ee
We recall that a strongly continuous, uniformly bounded  family of bounded operators $\{U(t,s)\}_{(t,s) \in \gD_0}$ is called a propagator
if \eqref{eq:1.2} is satisfied for $\cI_0$ and $\gD_0$ instead of $\cI$ and $\gD$, respectively. Roughly speaking, a propagator is a solution operator restricted to $\gD_0$ where
the assumption that $U(t,0)x$ should be  a strict solution is dropped. Obviously, the notion of a propagator is weaker
then that one of a solution operator. For its existence one needs
{{only the assumptions (S1) and (S2), see Theorem 4.4 and 4.5 in \cite{NeiStephZagr2016} or Theorem 3.3 \cite{NeiStephZagr2017}.}}
Of course, the propagator coincides with the solution operator of \cite{Yagi1990} if the assumptions (S1)-(S3) are satisfied and $\gb > \ga$.

{{By}} Proposition 3.8 of \cite{NeiStephZagr2018} and  \eqref{eq:1.11}
{{we immediately get}} that under the assumptions (S1)-(S3) and $\gb > 2\ga-1$ the estimate
\be\la{eq:1.14v}
\esssup_{(t,s)\in \gD_0}\|V_n(t,s) - U(t,s)\| \le \frac{R_\gb}{n^\gb}, \quad n \in \dN,
\ee
%
holds, where the constant $R_\gb$ is that one of \eqref{eq:1.11}.
Notice that the condition $\gb > 2\ga-1$ is weaker than $\gb > \ga$, i.e., if $\gb > \ga$, then $\gb > 2\ga-1$ holds.
If $\ga$ satisfies the condition $\tfrac{1+\gb}{2} > \ga > \gb$, then the assumptions (I), (VI) and (VII) of \cite{Yagi1990} for the family $\{C(t)\}_{\in\cI}$
are not valid but nevertheless we get an approximation of the corresponding propagator $\{U(t,s)\}_{(t,s)\in\gD_0}$. 


The results are stronger than those in \cite{NeiStephZagr2016,NeiStephZagr2017} for Banach spaces.
In \cite{NeiStephZagr2016} a convergence rate $O(n^{-(\gb-\ga)})$ was found, whereas in
\cite{NeiStephZagr2017} the Lipschitz case has been considered and the rate $O(n^{-(1-\ga)})$ for $\ga \in (\tfrac{1}{2},1)$ was proved.

It turns out that the result \eqref{eq:1.8v} can be hardly improved. Indeed in \cite{NeiStephZagr2018b} the simple case
$\gotH := \dC$ and $A = 1$ was considered. In that case the family $\{B(t)\}_{t\in\cI}$ reduces to a non-negative
bounded measurable function: $b(\cdot): \cI \longrightarrow \dR$ which has to be H\"older continuous with exponent
$\gb \in (0,1)$. For that case it was found in \cite{NeiStephZagr2018b} that the convergence rate is $O(n^{-\gb})$
which coincides with \eqref{eq:1.8v}. For the Lipschitz case it was found $O(n^{-1})$ which suggests that \eqref{eq:1.9v} and \eqref{eq:1.9vv} might be {{not}} optimal.\\

The paper is organised as follows. In Section \ref{sec:II} we give a short introduction into evolution semigroups. For more details the reader is referred to \cite{Nei1981,NeiStephZagr2016,
NeiZag2009,Nickel1996}. The results are proven in Section \ref{sec:III}. In Section \ref{sec:III.1} we prove auxiliary results which are necessary to prove the main results of Section
\ref{sec:III.2}.\\

{\bf Notation:} Spaces, in particular, Hilbert are denoted by Gothic capital letters like $\gotH$, $\gotK$, etc.
Operators are denoted by Latin or italic capital letters. The Banach space of bounded operators on space is denoted by $\cL(\cdot)$, like $\cL(\gotH)$.
 We set $\dR_+ := [0,\infty)$. If a function is called measurable, then it means Lebesgue measurable.
The notation ``a.e.'' means that a statement or relation fails at most for a set of Lebesgue
measure zero. In the following we use the notation $\esssup_{(t,s)\in \gD}$ or $\esssup_{(t,s)\in \gD_0}$.
In that case the Lebesgue measure of $\dR^2$ is meant.

We point out that we call operator $\cK$ to be generator of a semigroup $\{e^{-\tau \cK}\}_{\tau\in\dR_+}$, see e.g. \cite{ReedSim-II1975},
although in \cite{EngNag2000, Kato1980} it is the operator $ - \cK$, which is called the generator.

\section{Evolution semigroups}\la{sec:II}

Below we consider the Hilbert space
$\gotK = L^2(\cI,\gotH)$ consisting of all measurable functions $f(\cdot): \cI \longrightarrow \gotH$
such that the norm function $\|f(\cdot)\|: \cI \longrightarrow \dR_+$ is square integrable.
Further, let $D_0$ be the generator of the right-hand sift semigroup on $L^2(\cI,\gotH)$, i.e.
\bed
(e^{-\gt D_0}f)(t) = \chi_\cI(t-\gt)f(t-\gt), \quad t \in \cI, \quad \gt \in \dR_+, \quad f \in L^2(\cI,\gotH).
\eed
Notice that $e^{-\gt D_0} = 0$ for $\gt \ge T$. The generator $D_0$ is given by
\be\la{eq:2.3-1}
\begin{split}
(D_0f)(t) :=& \frac{\partial}{\partial t}f(t),\\
f \in \dom(D_0) :=& W^{1,2}_0(\cI,\gotH) = \{f \in W^{1,2}(\cI,\gotH): f(0) = 0\}.
\end{split}
\ee
We remark that $D_0$ is a closure of the maximal symmetric operator and its semigroup
 is contractive.

Further we consider the multiplication operator $\cA$ in $L^2(\cI,\gotH)$,
\bed
\begin{split}
(\cA f)(t) :=& Af(t), \\
f \in \dom(\cA) :=&
\left\{
 f \in L^2(\cI,\gotH):
\begin{matrix}
&f(t) \in \dom(A) \quad \mbox{for a.e. $t \in \cI$}\\
& Af(t) \in L^2(\cI,\gotH).
\end{matrix}
\right\}
\end{split}
\eed
If (S1) is satisfied, then $\cA$ is self-adjoint and $\cA \ge I_{L^2(\cI,\gotH)}$. For the resolvent one has the representation
\bed
((\cA - z)^{-1}f)(t)= (A - z)^{-1}f(t), \; t \in \cI_0, \; f \in L^2(\cI,\gotH), \quad z \in \rho(\cA) = \rho(A),
\eed
and the corresponding semigroup $\{e^{-\gt \cA}\}_{\gt \in \dR_+}$ is given by
\be\la{eq:2.3-1-1}
(e^{-\gt \cA}f)(t) = e^{-\gt A}f(t), \; t \in \cI, \; f \in L^2(\cI,\gotH), \quad \gt \in \dR_+.
\ee
Notice that the operators $e^{-\gt D_0}$ and $e^{-\gt \cA}$ commute. Let us consider the contraction semigroup
\be\la{eq:2.3-1-2}
\cU_0(\gt) := e^{-\gt D_0} e^{-\gt \cA}, \quad \gt \in \dR_+.
\ee
Obviously, the semigroup $\{\cU_0(\tau)\}_{\gt \in \dR_+}$ admits the representation \eqref{eq:1.13v}.
Due to the maximal $L^2$-regularity of $A$, cf. \cite{Arendt2007}, its generator $\cK_0$
is given by
\bed
\cK_0 := D_0 + \cA, \quad \dom(\cK_0) := \dom(D_0) \cap \dom(\cA) .
\eed

Further we consider the multiplication operator $\cB$, defined as
\be\la{eq:2.3-2}
\begin{split}
(\cB f)(t) :=& B(t)f(t)\\
f \in \dom(\cB) :=& \left\{f\in L^2(\cI,\gotH):
\begin{matrix}
&f(t) \in \dom(B(t)) \quad \mbox{for a.e. $t \in \cI$}\\
& B(t)f(t) \in L^2(\cI,\gotH)
\end{matrix}
\right\}\; .
\end{split}
\ee
If (S1) is satisfied, then $\cB$ is self-adjoint and non-negative. For the resolvent we have the representation
\bed
((\cB - z)^{-1}f)(t) = (B(t) - z)^{-1}f(t), \quad f \in L^2(\cI,\gotH), \quad z \in \dC_\pm,
\eed
for a.e. $t \in \cI$. The semigroup $\{e^{-\gt \cB}\}_{\gt \in \dR_+}$, admits the representation
\bed
(e^{-\gt \cB}f)(t) = e^{-\gt B(t)} f(t), \quad f \in L^2(\cI,\gotH),
\eed
for a.e. $t \in \cI$.

By \cite[Proposition 4.4]{NeiStephZagr2016} we get that under the assumptions (S1) and (S2) the operator
\bed
\cK := \cK_0 + \cB, \quad \dom(\cK) := \dom(\cK_0) \cap \dom(\cB),
\eed
is a generator of a contraction semigroup on $L^2(\cI,\gotH)$. From \cite[Proposition 4.5]{NeiStephZagr2016} we obtain that $\cK$ is the generator of an evolution semigroup.
Because $\cK$ is a generator of a contraction semigroup it turns out that the corresponding
propagator consists of contractions.

If $\{U(t,s)\}_{(t,s)\in \gD_0}$ is a propagator, then by virtue of
\eqref{eq:1.13v} it defines a semigroup, which by definition is an evolution semigroup. It
turns out that there is a one-to-one correspondence between the set of
evolution semigroups on $L^2(\cI,\gotH)$ and propagators. It is interesting to note that
evolution generators can be characterize quite independent from a propagator,
see \cite[Theorem 2.8]{Nei1981} or \cite[Theorem 3.3]{NeiStephZagr2016}.

\section{Results}\la{sec:III}

We start with a general observation concerning the conditions (S1)-(S3).
\br\la{rem:2.2}
{\em
If the conditions (S1)-(S3) are satisfied for some $\ga \in [0,1)$, then they are also satisfied for
each $\alpha^\prime \in (\ga,1]$.
Indeed, the condition (S1) is obviously satisfied. To show (S2) we note that
$\dom(A^{\alpha^\prime}) \subseteq \dom(A^\ga) \subseteq \dom(B(t))$ for a.e. $t \in \cI$. Using the
representation
\be\la{eq:2.3}
B(t)A^{-{\alpha^\prime}} = B(t)A^{-\ga}A^{-({\alpha^\prime}-\ga)}
\ee
for a.e. $t \in \cI$ we get that the map $B(\cdot)A^{-{\alpha^\prime}}: \cI \longrightarrow \cL(\gotH)$
is strongly measurable. Further, from \eqref{eq:2.3}
\bed
C_{\alpha^\prime} := \esssup_{t\in\cI}\|B(t)A^{-{\alpha^\prime}}\| \le \esssup_{t\in\cI}\|B(t)A^{-\ga}\| =
C_\ga < \infty.
\eed
Moreover we have
\bed
\|A^{-{\alpha^\prime}}(B(t) - B(s))A^{-{\alpha^\prime}}\|
\le \|A^{-\ga}(B(t) - B(s))A^{-\ga}\| \le L_{\ga,\gb} |t-s|^\gb,
\eed
$(t,s) \in \cI \times \cI$, which shows that there is a constant $L_{{\alpha^\prime},\gb} \le L_{\ga,\gb}$
such that
\bed
\|A^{-{\alpha^\prime}}(B(t) - B(s))A^{-{\alpha^\prime}}\| \le L_{{\alpha^\prime},\gb} |t-s|^\gb \quad (t,s)
\in \cI \times \cI.
\eed
holds for $(t,s) \in \cI \times \cI$.\hfill$\triangle$
}
\er

Since $A$ is self-adjoint and non-negative, one has
$\|A^\gamma e^{-\tau A}\|\leq {1}/{\tau^\gamma}$ for any $\tau\in\dR_+$ and $\gamma\in [0,1]$. Then by virtue of \eqref{eq:2.3-1-1} and of \eqref{eq:1.10-1v}, \eqref{eq:2.3-1-2} one gets
the estimates
\be\label{EstimateAorK0ExpA}
 \|\cA^\gamma e^{-\tau \cA}\| = \|\overline{e^{-\tau \cA} \cA^\gamma}\| \leq \frac{1}{\tau^\gamma} \;\;\mathrm{and}\;\;\|\cA^\gamma e^{-\tau \cK_0}\| = \|\overline{ e^{-\tau \cK_0} \cA^\gamma}\| \leq
 \frac{1}{\tau^\gamma}.
\ee

\subsection{Auxiliary estimates}\la{sec:III.1}

In this section we prove a series of estimates necessary to establish \eqref{eq:1.11}.
The following lemma can be partially derived from \cite[Lemma 7.4]{NeiStephZagr2016}.

\bl\la{lem:3.3}
Let the assumptions (S1) and (S2) be satisfied. Then for any $\gamma \in [\ga,1)$
there is a  constants $\gL_{\gamma} \geq 1$ such that
\be\la{eq:3.1a}
\|\cA^{\gamma} e^{-\gt \cK}\| \le \frac{\gL_{\gamma}}{\gt^{\gamma}}
\quad \mbox{and} \quad
\|\overline{e^{-\gt \cK}\cA^{\gamma}}\| \le \frac{\gL_{\gamma}}{\gt^{\gamma}}, \qquad \gt > 0,
\ee
holds.
\el
\begin{proof}
The proof of the first estimate follows from Lemma 7.4 of \cite{NeiStephZagr2016} and Remark \ref{rem:2.2}.
The second estimate can be proved similarly as the first one. One has only to modify the proof of Lemma 7.4 of \cite{NeiStephZagr2016}
in a suitable manner and to apply again Remark \ref{rem:2.2}.
\eproof
\end{proof}
\br
{\em
Lemma 2.1 of \cite{IchinoseTamura1998} claims that for the Lipschitz case  the solution operator $\{U(t,s)\}_{(t,s) \in \gD}$ of \eqref{eq:1.1} admits the estimates
\bed
\sup_{(t,s)\in\gD}(t-s)^\gga \|A^\gga U(t,s)\| < \infty
\quad \mbox{and} \quad
\sup_{(t,s)\in\gD}(t-s)^\gga \| U(t,s)A^\gga\| < \infty
\eed
for $\gga \in [0,1]$. Proposition 2.1 of \cite{NeiStephZagr2018b} immediately yields that the corresponding evolution semigroup $\{\cU(\gt) = e^{-\gt\cK}\}_{\tau\in\dR_+}$
satisfies the estimates \eqref{eq:3.1a} for $\gga=1$. \hfill$\triangle$
}
\er

Now we set
\be\la{eq:3.1b}
\cT(\gt) = e^{-\gt\cK_0}e^{-\gt \cB}, \quad \gt \in \dR_+.
\ee
Notice that $\cT(\gt) = 0$ for $\gt \ge T$.
\bl\label{lem:3.6}
Let the assumptions (S1) and (S2) be satisfied. Then for any $\gga \in [\ga,1)$
the estimates
\be\la{eq:3.4}
\|\cA^{-\gga}(\cT(\gt) - \cU(\gt))\| \le 2 C_\gga \gt
 \quad \mbox{and} \quad
\|(\cT(\gt) - \cU(\gt))\cA^{-\gga}\| \le 2 C_\gga\gt,
\ee
hold for $\gt \ge 0$, where
\be\la{eq:3.4-1}
C_\gga := \esssup_{t\in \cI}\|B(t)A^{-\gga}\| .
\ee
\el
\begin{proof}
The proof of the first estimate follows from Lemma 7.6 of \cite{NeiStephZagr2016} and Remark \ref{rem:2.2}.
The specific constant $2 C_\gga $ is obtained  following carefully the proof of Lemma 7.6 of \cite{NeiStephZagr2016}.
The second estimate can be proved modifying the proof of the first
estimate in an obvious manner. \eproof
\end{proof}

\bl\label{lem:3.7}
Let the assumptions (S1)-(S3) be satisfied. Then for any $\gga \in [\ga,1)$ and $\gb \in (0,1)$ there is a constant $Z_{\gga,\gb} > 0$ such that
\be\la{eq:3.10a}
\|\cA^{-\gga}(\cT(\gt)- \cU(\gt))\cA^{-\gga}\| \le Z_{\gga,\gb} \gt^{1+ \varkappa}, \quad \gt \in \dR_+,
\ee
holds where $\varkappa := \min\{\gga,\gb\}$.
\el

\begin{proof}
We use the representation:
\bed
\begin{split}
\frac{d}{d\gs}e^{-(\gt-\gs)\cK}e^{-\gs\cK_0}e^{-\gs\cB} =&
e^{-(\gt-\gs)\cK}\left\{\cK e^{-\gs\cK_0} - e^{-\gs\cK_0}\cK_0 - e^{-\gs\cK_0}\cB\right\}e^{-\gs\cB}\\
=& e^{-(\gt-\gs)\cK}\left\{\cB e^{-\gs\cK_0} -e^{-\gs\cK_0}\cB\right\}e^{-\gs\cB}
\end{split}
\eed
%
%
%
%
%
%
which yields
\bed
\begin{split}
e^{-(\gt-\gs)\cK}&\left\{\cB e^{-\gs\cK_0} - e^{-\gs\cK_0}\cB\right\}e^{-\gs\cB}\\
=& (e^{-(\gt-\gs)\cK}-e^{-(\gt-\gs)\cK_0})\Big\{\cB e^{-\gs\cK_0} - e^{-\gs\cK_0}\cB\Big\}(e^{-\gs\cB}-I) +\\
&e^{-(\gt-\gs)\cK_0}\Big\{\cB e^{-\gs\cK_0} - e^{-\gs\cK_0}\cB\Big\}(e^{-\gs\cB}-I)+\\
&(e^{-(\gt-\gs)\cK} - e^{-(\gt-\gs)\cK_0})\Big\{\cB e^{-\gs\cK_0} - e^{-\gs\cK_0}\cB\Big\} +\\
&e^{-(\gt-\gs)\cK_0}\Big\{\cB e^{-\gs\cK_0} - e^{-\gs\cK_0}\cB\Big\}\; .
\end{split}
\eed
Hence, we obtain the identity
\bed
\begin{split}
\cA^{-\gga}&e^{-(\gt-\gs)\cK}\left\{\cB e^{-\gs\cK_0} - e^{-\gs\cK_0}\cB\right\}e^{-\gs\cB}\cA^{-\gga}\\
=& \cA^{-\gga}(e^{-(\gt-\gs)\cK}-e^{-(\gt-\gs)\cK_0})\Big\{\cB e^{-\gs\cK_0} - e^{-\gs\cK_0}\cB\Big\}(e^{-\gs\cB}-I)\cA^{-\gga} +\\
&e^{-(\gt-\gs)\cK_0}\cA^{-\gga}\Big\{\cB e^{-\gs\cK_0} - e^{-\gs\cK_0}\cB\Big\}(e^{-\gs\cB}-I)\cA^{-\gga}+\\
&\cA^{-\gga}(e^{-(\gt-\gs)\cK} - e^{-(\gt-\gs)\cK_0})\Big\{\cB e^{-\gs\cK_0} - e^{-\gs\cK_0}\cB\Big\}\cA^{-\gga} +\\
&e^{-(\gt-\gs)\cK_0}\cA^{-\gga}\Big\{\cB e^{-\gs\cK_0} - e^{-\gs\cK_0}\cB\Big\}\cA^{-\gga}
\end{split}
\eed
which leads to the estimate
\begin{align}
\Big\|&\cA^{-\gga}e^{-(\gt-\gs)\cK}\left\{\cB e^{-\gs\cK_0} - e^{-\gs\cK_0}\cB\right\}e^{-\gs\cB}\cA^{-\gga}\Big\| \le \nonumber\\
& \Big\|\cA^{-\gga}(e^{-(\gt-\gs)\cK}-e^{-(\gt-\gs)\cK_0})\Big\{\cB e^{-\gs\cK_0} - e^{-\gs\cK_0}\cB\Big\}(e^{-\gs\cB}-I)\cA^{-\gga}\Big\| +\label{eq:3.5}\\
&\Big\|e^{-(\gt-\gs)\cK_0}\cA^{-\gga}\Big\{\cB e^{-\gs\cK_0} - e^{-\gs\cK_0}\cB\Big\}(e^{-\gs\cB}-I)\cA^{-\gga}\Big\|+\label{eq:3.6}\\
&\Big\|\cA^{-\gga}(e^{-(\gt-\gs)\cK} - e^{-(\gt-\gs)\cK_0})\Big\{\cB e^{-\gs\cK_0} - e^{-\gs\cK_0}\cB\Big\}\cA^{-\gga}\Big\| +\label{eq:3.7}\\
&\Big\|e^{-(\gt-\gs)\cK_0}\cA^{-\gga}\Big\{\cB e^{-\gs\cK_0} - e^{-\gs\cK_0}\cB\Big\}\cA^{-\gga}\Big\|\label{eq:3.8}
\end{align}

Note that by (\ref{EstimateAorK0ExpA}) and (\ref{eq:3.4-1}) one gets
\be\la{eq:3.8v}
\begin{split}
\|e^{-\gs\cK_0}\cB\| =& \| e^{-\gs\cK_0}\cA^\gga\overline{\cA^{-\gga}\cB}\| \le \gs^{-\gga} \|\overline{\cA^{-\gga}\cB}\| =  C_\gga \gs^{-\gga},\\
\|\cB e^{-\gs\cK_0}\| =& \| \cB \cA^{-\gga} \cA^\gga e^{-\gs\cK_0}\| \le \gs^{-\gamma} \|\cB\cA^{-\gga}\| = C_\gga \gs^{-\gga},
\end{split}
\ee
for $\gs > 0$. Due to \eqref{eq:3.8v} one estimates \eqref{eq:3.5} as
\bed
\begin{split}
\Big\|\cA^{-\gga}&(e^{-(\gt-\gs)\cK}-e^{-(\gt-\gs)\cK_0})\Big\{\cB e^{-\gs\cK_0} - e^{-\gs\cK_0}\cB\Big\}(e^{-\gs\cB}-I)\cA^{-\gga}\Big\| \\
\le & 2 \ C_\gga \gs^{-\gga} \, \Big\|\cA^{-\gga}(e^{-(\gt-\gs)\cK}-e^{-(\gt-\gs)\cK_0})\Big\|\Big\|(e^{-\gs\cB}-I)\cA^{-\gga}\Big\|.
\end{split}
\eed
Since the fundamental properties of semigroups and (\ref{eq:3.4-1}) yield
\be\la{eq:3.16}
\Big\|(e^{-\gs\cB}-I)\cA^{-\gga}\Big\| \le \|\cB \cA^{-\gga}\|\;\gs \le C_\gga \;\gs, \quad \gs \in \dR_+,
\ee
and
\bed
\Big\|\cA^{-\gga}(e^{-(\gt-\gs)\cK}-e^{-(\gt-\gs)\cK_0})\Big\| \le C_\gga(\gt-\gs), \quad \gs \in \dR_+,
\eed
we get for \eqref{eq:3.5} the estimate
\be\la{eq:3.13a}
\begin{split}
\Big\|\cA^{-\gga}&(e^{-(\gt-\gs)\cK}-e^{-(\gt-\gs)\cK_0})\Big\{\cB e^{-\gs\cK_0} - e^{-\gs\cK_0}\cB\Big\}(e^{-\gs\cB}-I)\cA^{-\gga}\Big\| \\
\le & 2\ C^3_\gga \gs^{1-\gga}(\gt-\gs), \quad 0 \le \gs \le \gt.
\end{split}
\ee

To estimate \eqref{eq:3.6} we recall that $\cl A$ and $\cl K_0$ commute. Then by (\ref{eq:3.4-1}) one gets
\be\la{eq:3.14a}
\begin{split}
\Big\|e^{-(\gt-\gs)\cK_0}&\cA^{-\gga}\Big\{\cB e^{-\gs\cK_0} - e^{-\gs\cK_0}\cB\Big\}(e^{-\gs\cB}-I)\cA^{-\gga}\Big\| \\
\le &2C_\gga\|(e^{-\gs\cB}-I)\cA^{-\gga}\| \le 2\   C^2_\gga \,\gs ,
\quad 0 \le \gs \le \gt\, ,
\end{split}
\ee
where \eqref{eq:3.16}  was used for the last inequality.

To estimate \eqref{eq:3.7} we have
\be\la{eq:3.15a}
\begin{split}
&\Big\|\cA^{-\gga}(e^{-(\gt-\gs)\cK} - e^{-(\gt-\gs)\cK_0})\Big\{\cB e^{-\gs\cK_0} - e^{-\gs\cK_0}\cB\Big\}\cA^{-\gga}\Big\|\\
\le & 2\  C_\gga \|\cA^{-\gga}(e^{-(\gt-\gs)\cK} - e^{-(\gt-\gs)\cK_0})\|
\le  2\  C^2_\gga  (\gt-\gs), \quad 0 \le \gs \le \gt.
\end{split}
\ee

To estimate \eqref{eq:3.8} we use the representation
\bed
\begin{split}
\cA^{-\gga}&\Big\{\cB e^{-\gs\cK_0} - e^{-\gs\cK_0}\cB\Big\}\cA^{-\gga}\\
 =& \cA^{-\gga}\cB (e^{-\gs \cA} - I)\cA^{-\gga}e^{-\gs\cD_0} - e^{-\gs D_0}\cA^{-\gga}(e^{-\gs \cA}-I)\cB\cA^{-\gga}  +\\
& \cA^{-\gga}\Big\{\cB e^{-\gs D_0}- e^{-\gs\cD_0}\cB\Big\}\cA^{-\gga} \ ,
\end{split}
\eed
that yields
\bed
\begin{split}
\Big\|\cA^{-\gga}&\Big\{\cB e^{-\gs\cK_0} - e^{-\gs\cK_0}\cB\Big\}\cA^{-\gga}\Big\|\\
 \le& \Big\|\cA^{-\gga}\cB (e^{-\gs \cA} - I)\cA^{-\gga}\Big\| + \Big\|\cA^{-\gga}(e^{-\gs \cA}-I)\cB\cA^{-\gga}\Big\|  +   \\
& \Big\|\cA^{-\gga}\Big\{\cB e^{-\gs D_0}- e^{-\gs\cD_0}\cB\Big\}\cA^{-\gga}\Big\|\;.
\end{split}
\eed
Then by (\ref{eq:3.4-1}) and by semigroup properties one gets
\bed
\Big\|\cA^{-\gga}\cB (e^{-\gs \cA} - I)\cA^{-\gga}\Big\| \le \frac{C_\gga }{\gga}\gs^\gga, \quad \gs \in \dR_+,
\eed
and
\bed
\Big\|\cA^{-\gga}(e^{-\gs \cA}-I)\cB\cA^{-\gga}\Big\| \le \frac{C_\gga }{\gga}\gs^\gga, \quad \gs \in \dR_+.
\eed
The last term is obtained by using (S3) (for $\alpha$ substituted by $\gamma$) and the
definitions \eqref{eq:2.3-1}, \eqref{eq:2.3-2}:
\bed
\Big\|\cA^{-\gga}\Big\{\cB e^{-\gs D_0}- e^{-\gs\cD_0}\cB\Big\}\cA^{-\gga}\Big\| \le \gs^\gb \,L_{\gga,\gb},
\quad \gs \in \dR_+.
\eed
Summing up one finds that
\be\la{eq:3.12}
\Big\|\cA^{-\gga}\Big\{\cB e^{-\gs\cK_0} - e^{-\gs\cK_0}\cB\Big\}\cA^{-\gga}\Big\| \le \frac{2C_\gga }{\gga}\gs^\gga + L_{\gga,\gb} \gs^\gb, \quad \gs \in \dR_+.
\ee

Using the estimates \eqref{eq:3.13a}, \eqref{eq:3.14a}, \eqref{eq:3.15a} and \eqref{eq:3.12} we get
\bed
\begin{split}
\Big\|\cA^{-\gga}e^{-(\gt-\gs)\cK}&\left\{\cB e^{-\gs\cK_0} - e^{-\gs\cK_0}\cB\right\}e^{-\gs\cB}\cA^{-\gga}\Big\|\\
\le & 2C^3_\gga \gs^{1-\gga}(\gt-\gs) + 2 C^2_\gga\,\gs + 2 C^2_\gga (\gt-\gs)+ \frac{2C_\gga }{\gga}\gs^\gga + L_{\gga,\gb} \gs^\gb\\
= & 2C^3_\gga \gs^{1-\gga}(\gt-\gs) + 2 C^2_\gga\,\gt + \frac{2C_\gga }{\gga}\gs^\gga + L_{\gga,\gb} \gs^\gb ,
\end{split}
\eed
or returning back to its derivative
\bed
\begin{split}
\Big\|\cA^{-\gga}\frac{d}{d\gs}&e^{-(\gt-\gs)\cK} e^{-\gs\cK_0}e^{-\gs\cB}\cA^{-\gga}\Big\|\\
\le & 2C^3_\gga \gs^{1-\gga}(\gt-\gs) + 2 C^2_\gga\,\gt + \frac{2C_\gga }{\gga}\gs^\gga +  L_{\gga,\gb} \gs^\gb, \quad 0 \le \gs \le \gt.
\end{split}
\eed
Since
\bed
\cA^{-\gga}(e^{-\gt\cK_0}e^{-\gt\cB} - e^{-\gt\cK})\cA^{-\gga} = \int^\gt_0 \cA^{-\gga}\frac{d}{d\gs}e^{-(\gt-\gs)\cK}e^{-\gs\cK_0}e^{-\gs\cB}\cA^{-\gga}d\gs ,
\eed
we find the estimate
\bed
\Big\|\cA^{-\gga}(e^{-\gt\cK_0}e^{-\gt\cB} - e^{-\gt\cK})\cA^{-\gga}\Big\| \le
\int^\gt_0 \Big\|\cA^{-\gga}\frac{d}{d\gs}e^{-(\gt-\gs)\cK} e^{-\gs\cK_0}e^{-\gs\cB}\cA^{-\gga}\Big\|d\gs ,
\eed
which yields the estimate
\bed
\begin{split}
\Big\|\cA^{-\gga}&(e^{-\gt\cK_0}e^{-\gt\cB} - e^{-\gt\cK})\cA^{-\gga}\Big\| \\
\le & 2C^3_\gga \int^\gt_0 \gs^{1-\gga}(\gt -\gs)d\gs + 2C^2_\gga \gt^2  + \frac{2C_\gga }{(1+\gga)\gga}\gt^{1+\gga} + \frac{L_{\gga,\gb}}{1+\gb}\gt^{1+\gb}
\end{split}
\eed
or after integration:
\bed
\begin{split}
\Big\|\cA^{-\gga}&(e^{-\gt\cK_0}e^{-\gt\cB} - e^{-\gt\cK})\cA^{-\gga}\Big\| \\
\le & \frac{2C^3_\gga}{(2-\gga)(3-\gga)} \gt^{3-\gga} + 2C^2_\gga \gt^2  + \frac{2C_\gga }{(1+\gga)\gga}\gt^{1+\gga} + \frac{L_{\gga,\gb}}{1+\gb}\gt^{1+\gb},
\quad \gt \in \dR_+ \ .
\end{split}
\eed
If $\gga \in [\ga,1)$ and $\gga \le \gb < 1$, then one gets
\begin{align}
\Big\|\cA^{-\gga}&(e^{-\gt\cK_0}e^{-\gt\cB} - e^{-\gt\cK})\cA^{-\gga}\Big\| \la{eq:3.18v}  \\
\le & \Big(\frac{2C^3_\gga}{(2-\gga)(3-\gga)} \gt^{2-2\gga} + 2C^2_\gga\gt^{1-\gga} + \frac{2C_\gga }{(1+\gga)\gga} + \frac{L_{\gga,\gb}}{1+\gb}\gt^{\gb-\gga}\Big)\gt^{1+\gga},\nonumber
\end{align}
$\gt \in \dR_+$, which immediately yields \eqref{eq:3.10a}.

If $\gga \in [\ga,1)$ and $0 < \gb < \gga$, then one can rewrite it as
\begin{align}
\Big\|\cA^{-\gga}&(e^{-\gt\cK_0}e^{-\gt\cB} - e^{-\gt\cK})\cA^{-\gga}\Big\| \la{eq:3.19v}\\
\le & \left(\frac{2C^3_\gga}{(2-\gga)(3-\gga)} \gt^{2-\gga-\gb} + 2C^2_\gga \gt^{1-\gb}  + \frac{2C_\gga }{(1+\gga)\gga}\gt^{\gga-\gb} + \frac{L_{\gga,\gb}}{1+\gb}\right)\gt^{1+\gb},\nonumber
\end{align}
$\gt \in \dR_+$, which shows \eqref{eq:3.10a} for this choice of $\gamma$ and $\beta$.
\eproof
\end{proof}
\br
{\em
For $\gga \in [\ga,1)$ and $\gga \le \gb < 1$
we find from \eqref{eq:3.18v} that
\be\la{eq:3.11a}
Z_{\gga,\gb} := \frac{2C^3_\gga}{(2-\gga)(3-\gga)} T^{2-2\gga} + 2C^2_\gga T^{1-\gga} + \frac{2C_\gga }{(1+\gga)\gga} + \frac{L_{\gga,\gb}}{1+\gb}T^{\gb-\gga}\;.
\ee
For $\gga \in [\ga,1)$ and $0 < \gb < \gga$ we get from \eqref{eq:3.19v} that
\bed
Z_{\gga,\gb} := \frac{2C^3_\gga}{(2-\gga)(3-\gga)} T^{2-\gga-\gb} + 2C^2_\gga T^{1-\gb}  + \frac{2C_\gga }{(1+\gga)\gga}T^{\gga-\gb} + \frac{L_{\gga,\gb}}{1+\gb} \ .
\eed
Here $C_\gga := \esssup_{t\in\cI}\|\cB \cA^{-\gga}\|$, see (\ref{eq:3.1a}), and $L_{\gga,\gb}$ is the H\"older
constant of the function $A^{-\gga}B(\cdot)A^{-\gga}: \cI \longrightarrow \cL(\gotH)$, see (S3).
}
\er
\bl\la{lem:5.1}
Let the assumptions (S1) and (S2) be satisfied. Then
\be\la{eq:3.19}
\|\cA^\gga(\cU(\gt) - \cT(\gt))\cA^{-\gga}\| \le
\left(\tfrac{\gL_\gga}{1-\gga} + 1\right)C_\gga \gt^{1-\gga},
\quad \gt \in \dR_+,
\ee
for $\gga \in [\ga,1)$.
\el
\begin{proof}
We use the representation
\bed
\cU(\gt) - \cT(\gt)
= e^{-\gt\cK} - e^{-\gt\cK_0} + e^{-\gt\cK_0}(I - e^{-\gt\cB})
\eed
which yields
\bed
\cA^\gga(\cU(\gt) - \cT(\gt))\cA^{-\gga}
= \cA^\gga(e^{-\gt\cK} - e^{-\gt\cK_0})\cA^{-\gga} +
\cA^\gga e^{- \tau \cK_0}(I - e^{-\gt\cB})\cA^{-\gga}
\eed

Using
%
%
%
%
the semigroup property we obtain for the first term the representation:
\bed
\cA^\gga(e^{-\gt\cK} - e^{-\gt\cK_0})\cA^{-\gga}  = -\int^\gt_0 \cA^\gga e^{-(\gt-x)\cK}\cB\cA^{-\gga}
e^{-x\cK_0} dx \ .
\eed
Hence, by (\ref{eq:3.1a}) and (\ref{eq:3.4-1}) one gets
\begin{eqnarray}\label{eq:3.19-1}
&&\|\cA^\gga(e^{-\gt\cK} - e^{-\gt\cK_0})\cA^{-\gga}\| \le
\int^\gt_0 \|\cA^\gga e^{-(\gt-x)\cK}\| \|\cB\cA^{-\gga}\| dx \le \nonumber \\
&& \gL_\gga C_\gga \int^\gt_0 \frac{1}{(\gt-x)^\gga} dx  = \frac{\gL_\gga C_\gga}{1-\gga}\gt^{1-\gga}
\ .
\end{eqnarray}
To estimate the second term we use the inequality
\begin{equation*}
\|\cA^\gga e^{-{\gt}\cK_0}(I - e^{-\gt\cB})\cA^{-\gga}\|
\leq \|\cA^\gga e^{- \tau \cK_0}\| \|(I - e^{-\gt\cB})\cA^{-\gga}\| .
\end{equation*}
Using (\ref{EstimateAorK0ExpA}) and (\ref{eq:3.16}) we estimate the second term as
\begin{equation}\label{eq:3.19-2}
\|\cA^\gga e^{- \tau\cK_0}(I - e^{-\gt\cB})\cA^{-\gga}\| \leq
C_\gga \ \tau^{1- \gamma} \ .
\end{equation}
Now the estimates (\ref{eq:3.19-1}) and (\ref{eq:3.19-2}) yield \eqref{eq:3.19}. 
\eproof
\end{proof}
\bl\la{lem:5.2}
Let the assumption (S1) be satisfied. If for each $\gga \in [\ga,1)$ there is a constant $M_\gga > 0$ such that
\be\la{eq:3.20}
\|\cA^\gga \cT(\gt)^m\| \le \frac{M_\gga}{(m \gt)^\gga}, \quad m \in \dN,\quad \gt \in \dR_+,
\ee
holds for $\cT(\gt)$ defined in (\ref{eq:3.1b}), then
\be\la{eq:3.21}
\|\cA^\gs \cT(\gt)^m\| \le \frac{M^\gd_\gga}{(m \gt)^\gs}, \quad m \in \dN,
\ee
holds for $\gs \in [0,\gga]$ and $\gd: = {\gs}/{\gga}$.
\el
\begin{proof}
If \eqref{eq:3.20} is satisfied, then
\bed
\|\overline{(\cT(\gt)^*)^m\cA^\gga} \| \le \frac{M_\gga}{(m \gt)^\gga}, \quad m \in \dN,
\eed
holds, which is equivalent to
\bed
\cA^\gga \cT(\gt)^m\overline{(\cT(\gt)^*)^m\cA^\gga }\le \frac{M^2_\gga}{(m \gt)^{2\gga}}, \quad m \in \dN,
\eed
or
\bed
\cT(\gt)^m(\cT(\gt)^*)^m\le \frac{M^2_\gga}{(m \gt)^{2\gga}}\cA^{-2\gga}, \quad m \in \dN.
\eed
Let $\gd \in (0,1)$. Using the Heinz inequality \cite[Theorem X.4.2]{BirSolom1987} we get
\bed
\Big(\cT(\gt)^m(\cT(\gt)^*)^m\Big)^\gd\le \frac{M^{2\gd}_\gga}{(m \gt)^{2\gd\gga}}\cA^{-2\gd\gga}, \quad m \in \dN.
\eed
Since $\cT(\gt)^m(\cT(\gt)^*)^m$ is a self-adjoint contraction we get
\bed
\cT(\gt)^m(\cT(\gt)^*)^m \le \Big(\cT(\gt)^m(\cT(\gt)^*)^m\Big)^\gd, \quad m \in \dN,
\eed
which yields
\bed
\cT(\gt)^m(\cT(\gt)^*)^m \le \frac{M^{2\gd}_\gga}{(m \gt)^{2\gd\gga}}\cA^{-2\gd\gga}, \quad m \in \dN,
\eed
or
\bed
\cA^{\gd\gga}\cT(\gt)^m\overline{(\cT(\gt)^*)^m\cA^{\gd\gga}} \le \frac{M^{2\gd}_\gga}{(m \gt)^{2\gd\gga}}, \quad m \in \dN.
\eed
Therefore, one gets
\bed
\Big\|\overline{(\cT(\gt)^*)^m\cA^{\gd\gga}}\Big\| \le \frac{M^{\gd}_\gga}{(m \gt)^{\gd\gga}}, \quad m \in \dN,
\eed
or
\bed
\Big\|\cA^{\gd\gga}\cT(\gt)^m\Big\| \le \frac{M^{\gd}_\gga}{(m \gt)^{\gd\gga}}, \quad m \in \dN.
\eed
Setting $\gd = {\gs}/{\gga}$ we obtain the proof of \eqref{eq:3.21}.
\eproof
\end{proof}
\bl\la{lem:5.3}
Let the assumptions (S1) and (S2) be satisfied and let $\gga \in (\ga,1)$. Then there is a constant
$M_\gga > 0$ such that
\be\la{eq:5.0b}
\|\cA^\gga \cT(\gt)^m\| \le \frac{M_\gga}{(m \gt)^\gga}, \quad m= 1,2,\ldots n,
\ee
holds for any $T > 0 \ $ if $\ \gt \in (0, \tfrac{T}{n})$ and $n \geq n_0$ where $ n_0 := \lfloor(2(\frac{\gL_\gga}{1-\gga}+1) C_\gga)^{\tfrac{1}{1-\gga}} T\rfloor +1$ and $\lfloor x\rfloor$ denotes the largest integer smaller than $x$.
\el
\begin{proof}
Let $M_\gga > 0$ be a constant which satisfies the inequality
\be\la{eq:5.0a}
5\gL_\gga + 2\left(\frac{\gL_\gga}{1-\gga}+1\right)C_\gga M_\gga T^{1-\gga} \frac{1}{n^{1-\gga}}
+ 4\gL_\gga C_\gga M_\gga^{\tfrac{\ga}{\gga}}B(1-\ga,1-\gga) \; T^{1-\ga} \le M_\gga
\ee
for $n \geq n_0$.
Here constants $\gL_\gga$ and $C_\gga$ are
defined by Lemma \ref{lem:3.3} and Lemma \ref{lem:3.6}, respectively, while $B(\cdot,\cdot)$ denotes the  Euler Beta-function. (Note that such $M_\gga > 0$ always exists, see Remark \ref{rem:3.9} below.)

Let $m=1$.  Then by (\ref{EstimateAorK0ExpA}) and (\ref{eq:3.1b}) we get
\bed
\|\cA^\gga \cT(\gt)\| \le \|\cA^\gga e^{-\gt \cK_0}\| \le \frac{1}{\gt^\gga} \le \frac{\gL_\gga}{\gt^\gga} \le \frac{M_\gga}{\gt^\gga},
\eed
for $\gt > 0$ and, in particular, for $\gt \in (0, {T}/{n})$.  Hence \eqref{eq:5.0b} holds for $m=1$.

Let us assume that \eqref{eq:5.0b} holds for $l = 1,2,\ldots,m-1$, with $m \le n$, i.e.
\be\la{eq:3.25}
\|\cA^\gga \cT(\gt)^l\| \le \frac{M_\gga}{(l \gt)^\gga}, \quad l= 1,2,\ldots m-1,
\ee
for $\gt \in (0, {T}/{n})$. We are going to show that \eqref{eq:3.25} holds for $l=m$.  To this aim
we use the representation
\bed
\cU(\gt)^m - \cT(\gt)^m = \sum^{m-1}_{k=0}\cU(\gt)^{m-1-k}(\cU(\gt) - \cT(\gt))\cT(\gt)^k, \quad m = 2,3,\dots\,,
\eed
which implies
\bed
\cT(\gt)^m = \cU(\gt)^m - \sum^{m-1}_{k=0}\cU(\gt)^{m-1-k}(\cU(\gt) - \cT(\gt))\cT(\gt)^k, \quad m = 2,3,\dots\,.
\eed
Hence
\bed
\cA^\gga\cT(\gt)^m = \cA^\gga\cU(\gt)^m - \sum^{m-1}_{k=0}\cA^\gga\cU(\gt)^{m-1-k}(\cU(\gt) - \cT(\gt))\cT(\gt)^k
\eed
or
\bed
\begin{split}
\cA^\gga&\cT(\gt)^m = \cA^\gga\cU(\gt)^m - \cA^\gga\cU(\gt)^{m-1}(\cU(\gt) - \cT(\gt))\\
& - \cA^\gga(\cU(\gt) - \cT(\gt))\cT(\gt)^{m-1}-\sum^{m-2}_{k=1}\cA^\gga\cU(\gt)^{m-1-k}(\cU(\gt) - \cT(\gt))\cT(\gt)^k.
\end{split}
\eed
for $m = 3,4,\dots$. This yields the inequality
\begin{align}\la{eq:5.0}
\|&\cA^\gga\cT(\gt)^m\| \le  \|\cA^\gga\cU(\gt)^m\| + \|\cA^\gga\cU(\gt)^{m-1}(\cU(\gt) - \cT(\gt))\| + \\
& \| \cA^\gga(\cU(\gt) - \cT(\gt))\cT(\gt)^{m-1}\|+ \sum^{m-2}_{k=1}\|\cA^\gga\cU(\gt)^{m-1-k}(\cU(\gt) - \cT(\gt))\cT(\gt)^k\|\nonumber
\end{align}
for $m = 3,4,\dots$. From Lemma \ref{lem:3.3} we get the estimates
\bed
\|\cA^\gga\cU(\gt)^m\| \le \frac{\gL_\gga}{(m\gt)^\gga}, \quad m = 2,3,\dots\,,
\eed
and consequently:
\bed
\|\cA^\gga\cU(\gt)^{m-1}(\cU(\gt) - \cT(\gt))\| \le \frac{2\gL_\gga}{((m-1)\gt)^\gga} \le {{\frac{4\gL_\gga}{(m\gt)^\gga}}}, \quad m = 2,3,\dots\,.
\eed
Then summing up estimates for the first two terms in the right-hand side of (\ref{eq:5.0}) we obtain
\be\la{eq:5.1}
\|\cA^\gga\cU(\gt)^m\| + \|\cA^\gga\cU(\gt)^{m-1}(\cU(\gt) - \cT(\gt))\| \le \frac{5\gL_\gga}{(m\gt)^\gga},
\quad m = 2,3,\dots\,.
\ee

Next we get for the third term in the right-hand side of (\ref{eq:5.0}) the  estimate
\bed
\| \cA^\gga(\cU(\gt) - \cT(\gt))\cT(\gt)^{m-1}\| \le \| \cA^\gga(\cU(\gt) - \cT(\gt))\cA^{-\gga}\| \|\cA^\gga\cT(\gt)^{m-1}\|,
\eed
$m = 2,3,\dots$. Then using Lemma \ref{lem:5.1} we find that
\bed
\| \cA^\gga(\cU(\gt) - \cT(\gt))\cT(\gt)^{m-1}\| \le \left(\frac{\gL_\gga}{1-\gga}+1\right)\
C_\gga \gt^{1-\gga} \ \|\cA^\gga\cT(\gt)^{m-1}\|, \quad m = 2,3,\dots\,.
\eed
By assumption \eqref{eq:3.25} this yields
\bed
\| \cA^\gga(\cU(\gt) - \cT(\gt))\cT(\gt)^{m-1}\| \le \left(\frac{\gL_\gga}{1-\gga}+1\right) \ {M_\gga  C_\gga }\frac{1}{((m-1)\gt)^\gga}\gt^{1-\gga} \ , \quad m = 2,3,\dots\,,
\eed
for $\gt \in (0, {T}/{n})$, which leads to
\be\la{eq:5.2}
\| \cA^\gga(\cU(\gt) - \cT(\gt))\cT(\gt)^{m-1}\| \le \left(\frac{\gL_\gga}{1-\gga}+1\right)
 \ {M_\gga C_\gga} \frac{2}{(m\gt)^\gga}\gt^{1-\gga} \ , \quad m = 2,3,\dots\,.
\ee
%
%

Finally one gets for the sum in (\ref{eq:5.0})
\bed
\begin{split}
\sum^{m-2}_{k=1}\|&\cA^\gga\cU(\gt)^{m-1-k}(\cU(\gt) - \cT(\gt))\cT(\gt)^k\|\\
\le &
\sum^{m-2}_{k=1}\|\cA^\gga\cU(\gt)^{m-1-k}\| \|(\cU(\gt) - \cT(\gt))\cA^{-\ga}\| \|{\cA^\ga}\cT(\gt)^k\|, \quad m = 2,3,\dots\,.
\end{split}
\eed
Then by Lemma \ref{lem:3.3} this implies
\bed
\begin{split}
\sum^{m-2}_{k=1}\|&\cA^\gga\cU(\gt)^{m-1-k}(\cU(\gt) - \cT(\gt))\cT(\gt)^k\|\\
\le &
\ \gL_\gga \ \sum^{m-2}_{k=1} \frac{1}{((m-1-k)\gt)^\gga} \|(\cU(\gt) - \cT(\gt))\cA^{-\ga}\| \|{\cA^\ga}\cT(\gt)^k\|, \quad m = 2,3,\dots\,.
\end{split}
\eed
Taking into account Lemma \ref{lem:3.6} we get
\bed
\begin{split}
\sum^{m-2}_{k=1}\|&\cA^\gga\cU(\gt)^{m-1-k}(\cU(\gt) - \cT(\gt))\cT(\gt)^k\|\\
\le &
\ 2\gL_\gga C_\ga \ \sum^{m-2}_{k=1} \frac{\gt}{((m-1-k)\gt)^\gga}\|{\cA^\ga}\cT(\gt)^k\|,
\quad m = 2,3,\dots\,.
\end{split}
\eed
Finally, using assumption \eqref{eq:3.25} and Lemma \ref{lem:5.2} one obtains
\bed
\begin{split}
\sum^{m-2}_{k=1}\|&\cA^\gga\cU(\gt)^{m-1-k}(\cU(\gt) - \cT(\gt))\cT(\gt)^k\|\\
\le &
\ 2\gL_\gga C_\ga M_\gga^{\tfrac{\ga}{\gga}} \ \sum^{m-2}_{k=1} \frac{\gt}{((m-1-k)\gt)^\gga}\frac{1}{(k\gt)^\ga},
\quad m = 2,3,\dots\,,
\end{split}
\eed
or
\bed
\begin{split}
\sum^{m-2}_{k=1}\|&\cA^\gga\cU(\gt)^{m-1-k}(\cU(\gt) - \cT(\gt))\cT(\gt)^k\|\\
\le &
\ 2\gL_\gga C_\gga M_\gga^{\tfrac{\ga}{\gga}}\left(\sum^{m-2}_{k=1} \frac{1}{(m-1-k)^\gga}\frac{1}{k^\ga}\right)\gt^{1-\gga-\ga}, \quad m = 2,3,\dots\,,
\end{split}
\eed
for $\gt \in (0, {T}/{n})$. Since Lemma \ref{lem:3.9} below yields
\be\la{eq:5.6}
\sum^{m-2}_{k=1} \frac{1}{(m-1-k)^\gga}\frac{1}{k^\ga} \le B(1-\ga,1-\gga) (m-1)^{1-\gga-\ga}, \quad m = 2,3,\dots\,,
\ee
where $B(\cdot,\cdot)$ is the Euler Beta-function, we get
\bed
\begin{split}
\sum^{m-2}_{k=1}\|&\cA^\gga\cU(\gt)^{m-1-k}(\cU(\gt) - \cT(\gt))\cT(\gt)^k\|\\
\le &
2\gL_\gga C_\gga M_\gga^{\tfrac{\ga}{\gga}}B(1-\ga,1-\gga) \gt^{1-\gga-\ga} (m-1)^{1-\gga-\ga}, \quad m = 2,3,\dots\,,
\end{split}
\eed
which in turn leads to
\be\la{eq:5.3}
\begin{split}
\sum^{m-2}_{k=1}\|\cA^\gga&\cU(\gt)^{m-1-k}(\cU(\gt) - \cT(\gt))\cT(\gt)^k\|\\
\le&
\frac{4\gL_\gga C_\gga M_\gga^{\tfrac{\ga}{\gga}}B(1-\ga,1-\gga)}{(m\gt)^\gga} \gt^{1-\ga} m^{1-\ga},
\end{split}
\ee
for $m = 2,3,\dots$ and any $\gt \in (0,{T}/{n})$.

Now we take into account \eqref{eq:5.0}, \eqref{eq:5.1}, \eqref{eq:5.2}  and \eqref{eq:5.3}
to conclude that
\bed
\begin{split}
& \|\cA^\gga\cT(\gt)^m\| \le  \\
& \Big\{5\gL_\gga + 2\left(\frac{\gL_\gga}{1-\gga}+1\right)
{M_\gga C_\gga}\gt^{1-\gga}
+ 4\gL_\gga C_\gga M_\gga^{\tfrac{\ga}{\gga}}B(1-\ga,1-\gga) \gt^{1-\ga} m^{1-\ga}\Big\}\frac{1}{(m\gt)^\gga} \ ,
\end{split}
\eed
for $m = 2,3,\dots$ and $\gt \in (0, {T}/{n})$. Then
\bed
\begin{split}
& \|\cA^\gga\cT (\gt)^m\| \leq \\
& \Big\{5\gL_\gga + 2\left(\frac{\gL_\gga}{1-\gga}+1\right)
{M_\gga C_\gga}T^{1-\gga} \frac{1}{n^{1-\gga}}
+ 4\gL_\gga C_\gga M_\gga^{\tfrac{\ga}{\gga}}B(1-\ga,1-\gga) \; T^{1-\ga}\Big\}\frac{1}{(m\gt)^\gga}.
\end{split}
\eed
%
%
>From assumption \eqref{eq:5.0a} we get
\bed
5\gL_\gga +2\left(\frac{\gL_\gga}{1-\gga}+1\right)
{M_\gga C_\gga}T^{1-\gga} \frac{1}{n^{1-\gga}}+ 4\gL_\gga C_\gga M_\gga^{\tfrac{\ga}{\gga}}B(1-\ga,1-\gga) \; T^{1-\ga} \le M_\gga
\eed
for $n \geq n_0$, which shows that \eqref{eq:3.25} holds for
$l = 1,2,3,\dots,n$ and $n \ge n_0$ which proves \eqref{eq:5.0b}.
\eproof
\end{proof}
\br\la{rem:3.9}
{\em
One checks that condition \eqref{eq:5.0a} is always satisfied for sufficiently large $M =M_\gga $
and $n \geq n_0$.
Indeed, after setting
\bed
c_0 := 5\gL_\gga, \quad c_1 := 2\left(\frac{\gL_\gga}{1-\gga}+1\right) C_\gga T^{1-\gga}, \quad c_2 := 4\gL_\gga C_\gga B(1-\ga,1-\gga) \; T^{1-\ga}
\eed
we get the condition
\bed
c_0 + \frac{c_1}{n^{1-\gga}} M + c_2 M^{\tfrac{\ga}{\gga}} \le M
\eed
which yields
\bed
c_0 + c_2 M^{\tfrac{\ga}{\gga}} \le (1- \frac{c_1}{n^{1-\gga}})M
\eed
or
\bed
\frac{c_0}{M} + \frac{c_2}{M^{1-\tfrac{\ga}{\gga}}} \le 1- \frac{c_1}{n^{1-\gga}}
\eed
Since $n > c^{\tfrac{1}{1-\gga}}_1$ we have $1- {c_1}/{n^{1-\gga}} > 0$. The left-hand side tends to zero
if $M \to \infty$. Hence, choosing $M$ sufficiently large we guarantee
the existence of  $M_\gga$ such that condition \eqref{eq:5.0a} is satisfied for any
$n \geq n_0$.\hfill$\triangle$
}
\er
It remains only to verify the following statement.
\bl\la{lem:3.9}
Let $\ga \in [0,1)$ and $\gga \in [\ga,1)$. Then
\bed
\sum^{n-1}_{k=1} \frac{1}{(n-k)^\gga}\frac{1}{k^\ga} \le B(1-\ga,1-\gga) n^{1-\gga-\ga}, \quad n \in 2,3,\ldots\,.
\eed
the estimate holds where $B(\cdot,\cdot)$ is the Euler Beta-function.
\bed
B(1-\ga,1-\gga) := \int^1_0 \frac{1}{x^\ga (1-x)^\gga}dx
\eed
\el
\begin{proof}
If $x \in (k-1,k]$, then
\bed
\frac{1}{k^\ga} \le \frac{1}{x^\ga}
\quad \mbox{and} \quad
\frac{1}{(n-k)^\gga} \le \frac{1}{(n-1-x)^\gga}
\eed
for $k = 1,2,\ldots,n-1$. Hence
\bed
\frac{1}{(n-k)^\gga k^\ga} \le \frac{1}{(n-1-x)^\gga x^\ga}, \quad x \in (k-1,k].
\eed
Therefore
\bed
\frac{1}{(n-k)^\gga k^\ga} = \int^k_{k-1}\frac{1}{(n-k)^\gga k^\ga}dx \le \int^k_{k-1}\frac{1}{(n-1-x)^\gga x^\ga}dx, \quad x \in (k-1,k],
\eed
or
\bed
\begin{split}
\sum^{n-1}_{k=1}\frac{1}{(n-k)^\gga k^\ga} =& \sum^{n-1}_{k=1}\int^k_{k-1}\frac{1}{(n-k)^\gga k^\ga}dx
\le \sum^{n-1}_{k=1}\int^k_{k-1}\frac{1}{(n-1-x)^\gga x^\ga}dx\\
 =& \int^{n-1}_0 \frac{1}{(n-1-x)^\gga x^\ga}dx = B(1-\ga,1-\gga)n^{1-\ga-\gga}
\end{split}
\eed
\eproof
\end{proof}

\subsection{Main Results}\la{sec:III.2}

In this section we collect our main results and their proofs. They are based on preliminaries
established in Section \ref{sec:III.1}.

\begin{thm}\la{thm:3.11}
Let the assumptions (S1) -(S3) be satisfied and let $ \gb > 2\ga -1$. Then there is a
constant $R_\gb > 0$ such that
\be\la{eq:3.25v}
\sup_{\tau \in \dR_+}\|\cU(\tau) - \cT(\tau/n)^n\| \le \frac{R_\gb}{n^\gb}
\ee
holds for $n \in \dN$ and $\tau \in \dR_+$.
\end{thm}
\begin{proof}
%
%
Taking into account the representation
\bed
\cU(\tau/n)^n - \cT(\tau/n)^n = \sum^{n-1}_{m=0}\cU(\tau/n)^{n-m-1}(\cU(\tau/n) - \cT(\tau/n))\cT(\tau/n)^m, \quad n \in \dN,
\eed
or, identically,
\bed
\begin{split}
\cU(\tau/n)^n& - \cT(\tau/n)^n\\
= & \cU(\tau/n)^{n-1}(\cU(\tau/n) - \cT(\tau/n)) +
(\cU(\tau/n) - \cT(\tau/n))\cT(\tau/n)^{n-1}+\\
&\sum^{n-2}_{m=1}\cU(\tau/n)^{n-m-1}(\cU(\tau/n) - \cT(\tau/n))\cT(\tau/n)^m, \quad n=3,4,\ldots \ ,
\end{split}
\eed
we obtain the estimate
\begin{align}
\|\cU&(\tau/n)^n - \cT(\tau/n)^n\|\nonumber \\
\le & \;\|\cU(\tau/n)^{n-1}\cA^\gga\| \|\cA^{-\gga}(\cU(\tau/n) - \cT(\tau/n))\| \nonumber\\
& +\|(\cU(\tau/n) - \cT(\tau/n))\cA^{-\gga}\| \|\cA^{\gga}\cT(\tau/n)^{n-1}\|\la{eq:3.31}\\
&+\sum^{n-2}_{m=1}\|\cU(\tau/n)^{n-m-1}\cA^\gga\| \|\cA^{-\gga}(\cU(\tau/n) - \cT(\tau/n))\cA^{-\gga}\| \|\cA^\gga\cT(\tau/n)^m\|,\nonumber
\end{align}
$n =3,4,\ldots\,$.

Note that using Lemma \ref{lem:3.3} and Lemma \ref{lem:3.6} one gets
\bed
\|\cU(\tau/n)^{n-1}\cA^\gga\| \|\cA^{-\gga}(\cU(\tau/n) - \cT(\tau/n))\| \le 2 \ \frac{\gL_\gga C_\gga}
{({\tau (n-1)}/{n})^\gga}\ \frac{\tau}{n} \ ,
\eed
which yields
\be\la{eq:3.32}
\|\cU(\tau/n)^{n-1}\cA^\gga\| \|\cA^{-\gga}(\cU(\tau/n) - \cT(\tau/n))\| \le 2^{1+\gga}\gL_\gga C_\gga T^{1-\gga}\frac{1}{n}.
\ee
for $n = 3,4,\ldots$ and $\tau \in [0,T]$.

Now using Lemma \ref{lem:3.6} and Lemma \ref{lem:5.3} for $m=n-1$ we find
\bed
\|(\cU(\tau/n) - \cT(\tau/n))\cA^{-\gga}\| \|\cA^{\gga}\cT(\tau/n)^{n-1}\| \le 2 \ C_\gga \
\frac{\tau}{n} \ \frac{M_\gga}{(\tau(n-1)/{n})^\gga} \ ,
\eed
for $n \geq n_0$, where $n_0$ is defined in Lemma \ref{lem:5.3} and $\tau \in [0,T]$. Hence,
\be\la{eq:3.33}
\|(\cU(\tau/n) - \cT(\tau/n))\cA^{-\gga}\| \|\cA^{\gga}\cT(\tau/n)^{n-1}\| \le 2^{1+\gga} \ C_\gga M_\gga \ T^{1-\gga}
\frac{1}{n} \ .
\ee
%
%

Taking into account Lemma \ref{lem:3.3}, Lemma \ref{lem:3.7} and Lemma \ref{lem:5.3}
(for $\varkappa = \min \{\gamma, \beta\}$) one gets
\bed
\begin{split}
\sum^{n-2}_{m=1}\|\cU(\tau/n)^{n-m-1}\cA^\gga\| &\|\cA^{-\gga}(\cU(\tau/n) - \cT(\tau/n))\cA^{-\gga}\| \|\cA^\gga\cT(\tau/n)^m\|\\
\le & \sum^{n-2}_{m=1}\frac{\gL_\gga \ Z_{\gga,\gb}}{({(n-m-1)\ \tau}/{n})^\gga} \  \Big(\frac{\tau}{n}\Big)^{1+\varkappa}\frac{M_\gga}{(m \ {\tau}/{n})^\gga }\\
= &\frac{\gL_\gga Z_{\gga,\gb} M_\gga \tau^{1+\varkappa-2\gga}}{n^{1+\varkappa-2\gga}} \ \sum^{n-2}_{m=1}\frac{1}{(n-m-1)^\gga}\frac{1}{m^\gga},
\end{split}
\eed
for $n > \max\left\{2, n_0\right\}$ and $\tau \in [0,T]$ .
Then by \eqref{eq:5.6} we obtain
\bed
\begin{split}
\sum^{n-2}_{m=1}&\|\cU(\tau/n)^{n-m-1}\cA^\gga\| \|\cA^{-\gga}(\cU(\tau/n) - \cT(\tau/n))\cA^{-\gga}\| \|\cA^\gga\cT(\tau/n)^m\|\\
\le &\frac{\gL_\gga Z_{\gga,\gb} M_\gga \tau^{1+\varkappa-2\gga}}{n^{1+\varkappa-2\gga}}\
B(1-\gga,1-\gga) \,n^{1-2\gga} \ ,
\end{split}
\eed
or
\be\la{eq:3.34}
\begin{split}
\sum^{n-2}_{m=1}&\|\cU(\tau/n)^{n-m-1}\cA^\gga\| \|\cA^{-\gga}(\cU(\tau/n) - \cT(\tau/n))\cA^{-\gga}\| \|\cA^\gga\cT(\tau/n)^m\|\\
\le & \ \gL_\gga Z_{\gga,\gb} M_\gga B(1-\gga,1-\gga)\ T^{1+\varkappa-2\gga} \ \frac{1}{n^\varkappa} \ .
\end{split}
\ee
%
%

Therefore, by virtue of  \eqref{eq:3.31}, \eqref{eq:3.32}, \eqref{eq:3.33} and \eqref{eq:3.34} we get for $n > \max\left\{2,n_0\right\}$ and $\tau \in [0,T]$ the estimate
\bed
\begin{split}
\|\cU&(\tau)^n - \cT(\tau/n)^n\| = \|\cU(\tau/n)^n - \cT(\tau/n)^n\|\\
\le & 2^{1+\gga}\gL_\gga C_\gga T^{1-\gga}\frac{1}{n} + 2^{1+\gga}C_\gga M_\gga T^{1-\gga} \frac{1}{n} +\gL_\gga Z_{\gga,\gb} M_\gga B(1-\gga,1-\gga)T^{1+\varkappa-2\gga}\frac{1}{n^\varkappa}\\
\le &\left\{2^{1+\gga}\gL_\gga C_\gga T^{1-\gga} + 2^{1+\gga}C_\gga M_\gga T^{1-\gga} +\gL_\gga Z_{\gga,\gb} M_\gga B(1-\gga,1-\gga) T^{1+\varkappa-2\gga}\right\}\frac{1}{n^\varkappa} \ .
\end{split}
\eed
If $ \ga  < \gb < 1$, then we choose $\gga = \gb$, i.e., $\varkappa = \gb$ and $1+\varkappa-2\gga =
1-\gb \ge 0$. Setting
\bed
R'_\gb := 2^{1+\gb}\gL_\gb C_\gb T^{1-\gb} + 2^{1+\gb}C_\gb M_\gb T^{1-\gb} +\gL_\gb Z_{\gb,\gb} M_\gb B(1-\gb,1-\gb) T^{1-\gb}
\eed
one obtains the estimate
\be\la{eq:3.37v}
\|\cU(\tau)^n - \cT(\tau/n)^n\| \le \frac{R'_\gb}{n^\gb} \ ,
\ee
for $n > \max\left\{2,n_0\right\}$ and $\tau \in [0,T]$ .

Now let $0 < \gb \le \ga$. Since $1+ \gb -2\ga > 0$, there exists $\gga \in (\ga,1)$ such that
$1 +\gb-2\gga \ge 0$. Indeed, there is a $\varepsilon > 0$ verifying
$1+\gb- 2\ga > 2\varepsilon$. Setting $\gga = \ga + \varepsilon$ we get $1+\gb-2\gga > 0$. Notice that $\varkappa = \gb$. Then setting
\bed
R'_\gb := 2^{1+\gga}\gL_\gga C_\gga T^{1-\gga} + 2^{1+\gga}C_\gga M_\gga T^{1-\gga} +\gL_\gga Z_{\gga,\gb} M_\gga B(1-\gga,1-\gga) T^{1+\gb-2\gga} \ ,
\eed
we obtain \eqref{eq:3.37v} for $n > \max\left\{2,n_0\right\}$.

Both results immediately imply that
there is a constant $R_\gga$ such that \eqref{eq:3.25v} holds for $\tau \in [0,T]$ and $n \in \dN$. Finally, using $\cU(\tau) = 0$ and $\cT({\tau}/{n})^n = 0$ for $\tau \ge T$ we obtain
\eqref{eq:3.25} for any $\tau \in \dR_+$.
\eproof
\end{proof}

Now we set
\bed
\wt\cT(\tau) := e^{-\tau\cB}e^{-\tau\cK_0}, \quad \tau \in \dR_+.
\eed
\bc\la{cor:3.12}
Let the assumptions (S1) -(S3) be satisfied and $\gb > 2\ga -1$. Then there exists $\wt R_\gb > 0$
such that estimate
\be\la{eq:3.31a}
\sup_{\tau \in \dR_+}\|\cU(\tau) - \wt\cT(\tau/n)^n\| \le \frac{\wt R_\gb}{n^\gb}
\ee
holds for $n \in \dN$ and $\tau \in \dR_+$.
\ec
\begin{proof}
Notice that
\bed
\wt\cT({\tau}/{n})^{n+1} = e^{-{\tau}\cB /{n}} \, \cT({\tau}/{n})^n \, e^{-{\tau}\cK_0/{n}},
\quad \tau \in \dR_+, \quad n \in \dN.
\eed
Hence
\bed
\begin{split}
\cU((n+1)& \tau/n) - \wt\cT(\tau/n)^{n+1} = e^{-(n+1)\tau\cK/n}- e^{-\tau\cB/n}\cT(\tau/n)^ne^{-\tau\cK_0/n}\\
=& e^{-(n+1)\tau\cK/n} - e^{-\tau\cB/n}e^{-\tau\cK} e^{-\tau\cK_0/n} + e^{-\tau\cB/n}(\cU(\tau)-\cT(\tau/n)^n) e^{-\tau\cK_0/n}\\
=& (I - e^{-\tau\cB/n})e^{-\tau\cK} e^{-\tau\cK_0/n} + e^{-\tau \cK}(e^{-\tau\cK/n} - e^{-\tau\cK_0/n}) +\\
& e^{-\tau\cB/n}(\cU(\tau)-\cT(\tau/n)^n) e^{-\tau\cK_0/n}, \quad \tau \in \dR_+, \quad n \in \dN,
\end{split}
\eed
which yields the estimate
\be\la{eq:3.34a}
\begin{split}
\|\cU((n+1)&\tfrac{\tau}{n}) - \wt\cT(\tfrac{\tau}{n})^{n+1}\| \\
\le & \|(I - e^{-\tfrac{\tau}{n}\cB})e^{-\tau\cK}\|  + \|e^{-\tau \cK}(e^{-\tfrac{\tau}{n}\cK} - e^{-\tfrac{\tau}{n}\cK_0})\| +\\
& \|\cU(\tau)-\cT(\tfrac{\tau}{n})^n\|, \quad \tau \in \dR_+, \quad n \in \dN.
\end{split}
\ee

Obviously, one has
\bed
\|(I - e^{-\tfrac{\tau}{n}\cB})e^{-\tau\cK}\|\le \|(I - e^{-\tfrac{\tau}{n}\cB})\cA^{-\ga}\|\|\cA^\ga e^{-\tau\cK}\|,\quad \tau \in \dR_+, \quad n \in \dN.
\eed
Using
\bed
(I - e^{-\tfrac{\tau}{n}\cB})\cA^{-\ga} = \int^{\tfrac{\tau}{n}}_0 e^{-\gs \cB}\cB\cA^{-\ga}d\gs, \quad \tau \in \dR_+, \quad n \in \dN,
\eed
we get the estimate
\bed
\|(I - e^{-\tfrac{\tau}{n}\cB})\cA^{-\ga}\| \le C_\ga \frac{\tau}{n}, \quad \tau \in \dR_+, \quad n \in \dN.
\eed
Taking into account condition (S2) and Lemma \ref{lem:3.3} we find
\be\la{eq:3.38}
\|(I - e^{-\tfrac{\tau}{n}\cB})e^{-\tau\cK}\|\le C_\ga \gL_\ga \frac{\tau^{1-\ga}}{n} \le C_\ga \gL_\ga T^{1-\ga}\frac{1}{n}, \quad \tau \in \dR_+, \quad n \in \dN,
\ee
where we have used that $e^{-\tau\cK} = 0$ for $\tau \ge T$.

Further, we have
\bed
\|e^{-\tau \cK}(e^{-\tfrac{\tau}{n}\cK} - e^{-\tfrac{\tau}{n}\cK_0})\| \le \|e^{-\tau \cK}\cA^\ga\|\,\|\cA^{-\ga}(e^{-\tfrac{\tau}{n}\cK} - e^{-\tfrac{\tau}{n}\cK_0})\|,
\eed
$\tau \in \dR_+$, $n \in \dN$. Then using
\bed
\cA^{-\ga}(e^{-\tfrac{\tau}{n}\cK} - e^{-\tfrac{\tau}{n}\cK_0}) = -\int^{\tfrac{\tau}{n}}_0 e^{-\gs \cK_0}\overline{\cA^{-\ga}\cB}e^{-(\tau-\gs)\cK}d\gs,
\eed
$\tau \in \dR_+$, $n \in \dN$, we find the estimate
\bed
\|\cA^{-\ga}(e^{-\tfrac{\tau}{n}\cK} - e^{-\tfrac{\tau}{n}\cK_0})\| \le  C_\ga \ \frac{\tau}{n}, \quad \tau \in \dR_+,
\quad n \in \dN.
\eed
Applying again Lemma \ref{lem:3.3} one gets
\be\la{eq:3.42}
\|e^{-\tau \cK}(e^{-\tfrac{\tau}{n}\cK} - e^{-\tfrac{\tau}{n}\cK_0})\|  \le C_\ga \gL_\ga T^{1-\ga} \
\frac{1}{n}, \quad \tau \in \dR_+, \quad n \in \dN.
\ee
The insertion of \eqref{eq:3.38} and \eqref{eq:3.42} into \eqref{eq:3.34a} yields
\bed
\|\cU((n+1)\tfrac{\tau}{n}) - \wt\cT(\tfrac{\tau}{n})^{n+1}\|  \le 2C_\ga\gL_\ga \ \frac{1}{n} + \|\cU(\tau)-\cT(\tfrac{\tau}{n})^n)\|, \quad \tau \in \dR_+, \quad n \in \dN.
\eed
Then by  Theorem \ref{thm:3.11} we obtain
\bed
\|\cU((n+1)\tfrac{\tau}{n}) - \wt\cT(\tfrac{\tau}{n})^{n+1}\|  \le 2C_\ga\gL_\ga \ \frac{1}{n} + R_\gga \
\frac{1}{n^\gga}, \quad \tau \in \dR_+, \quad n \in \dN.
\eed
Therefore, by setting $R'_\gga := 2C_\ga\gL_\ga + R_\gga$ we obtain
\bed
\|\cU((n+1)\tfrac{\tau}{n}) - \wt\cT(\tfrac{\tau}{n})^{n+1}\|  \le \frac{R'_\gga}{n^\gga}, \quad \tau \in \dR_+, \quad n \in \dN.
\eed
which yields
\bed
\sup_{\tau \in\dR_+}\|\cU((n+1)\tfrac{\tau}{n}) - \wt\cT(\tfrac{\tau}{n})^{n+1}\|  \le \frac{R'_\gga}{n^\gga}, \quad \tau \in \dR_+, \quad n \in \dN.
\eed
Let $\tau = \tau'{n}/{(n+1)}$ for $\tau' \in \dR_+$. Then
\bed
\sup_{\tau \in\dR_+}\|\cU((n+1)\tfrac{\tau}{n}) - \wt\cT(\tfrac{\tau}{n})^{n+1}\| = \sup_{\tau' \in\dR_+}\|\cU(\tau') - \wt\cT(\tfrac{\tau'}{n+1})^{n+1}\| \le \frac{R'_\gga}{n^\gga} \ ,
\eed
or
\bed
\sup_{\tau' \in\dR_+}\|\cU(\tau') - \wt\cT(\tfrac{\tau'}{n+1})^{n+1}\| \le 2^\gga \ \frac{R'_\gga}{(n+1)^\gga},
\eed
$\tau \in \dR_+$, $n \in \dN$. Setting $\wt R_\gga := \max\{2,2^\gga R'_\gga\}$ we prove \eqref{eq:3.31a}.

\eproof
\end{proof}
These results can be immediately extended to propagators. To this end we set
\be\la{eq:3.49}
\begin{split}
\wt G_j(t,s;n) :=& e^{-\tfrac{t-s}{n}B(t_j)}e^{-\tfrac{t-s}{n}A}, \quad j = 0,1,2,\ldots,n,\\
\wt V_n(t,s) :=& \wt G_n(t,s;n) \wt G_{n-1}(t,s;n) \times \cdots \times \wt G_2(t,s;n) \wt G_1(t,s;n),
\end{split}
\ee
$t_j := s + j\tfrac{t-s}{n}$, $j = 0,1,2,\ldots,n$, in analogy to \eqref{eq:1.6}.
\bt\la{thm:3.12}
Let the assumptions (S1)-(S3) be satisfied. Further, let $\{U(t,s)\}_{(t,s)\in\gD_0}$ be the propagator corresponding to the evolution generator $\cK$
and let $\{V_n(t,s)\}_{(t,s) \in \gD_0}$ and $\{\wt V_n(t,s)\}_{(t,s) \in \gD_0}$ be defined by \eqref{eq:1.6} and \eqref{eq:3.49}, respectively.
If $\gb > 2\ga -1$, then the estimates
\be\la{eq:3.30}
\esssup_{(t,s)\in\gD_0}\| U(t,s) - V_n(t,s)\| \le \frac{R_\gb}{n^\gb}
\quad \mbox{and} \quad
\esssup_{(t,s)\in\gD_0}\| U(t,s) - \wt V_n(t,s)\| \le \frac{\wt R_\gb}{n^\gb}
\ee
hold for $n \in \dN$, where the constants $R_\gga$ and $\wt R_\gga$ are those of Theorem \ref{thm:3.11} and Corollary \ref{cor:3.12}.
\et
\begin{proof}
Note that Proposition 2.1 of \cite{NeiStephZagr2018b} yields
\bed
\sup_{\tau \in \dR_+}\|\cU(\tau) - \cT(\tfrac{\tau}{n})^n\| = \esssup_{(t,s)\in\gD_0}\|U(t,s) - V_n(t,s)\|, \quad n \in \dN.
\eed
Then applying Theorem \ref{thm:3.11} we prove \eqref{eq:3.30}.

To proof the second estimate we use Proposition 3.8 of \cite{NeiStephZagr2018} where the relation
\bed
\sup_{\tau \in \dR_+}\|\cU(\tau) - \wt\cT(\tfrac{\tau}{n})^n\| = \esssup_{(t,s)\in\gD_0}\|U(t,s) - \wt V_n(t,s)\|, \quad n \in \dN.
\eed
was shown. Applying Corollary \ref{cor:3.12} we complete the proof.
\eproof
\end{proof}

\section{Example}

As an example we consider the diffusion equation perturbed by a time-dependent scalar
potential. For this aim let $\gotH=L^2(\Omega)$, where $\Omega\subset \mathbb{R}^3$
is a bounded domain with sufficiently smooth boundary. Domains in higher dimension can be treated analogously. The equation 
reads as
\begin{align}
 \dot u(t)=\Delta u(t) -B(t)u(t), \quad u(s)=u_s\in \gotH, \quad t,s\in [0,T] \ ,
 \label{EvolProbLaplaceAndTimeDependentPotential}
\end{align}
where $\Delta$ denotes the Laplace operator in $L^2(\Omega)$ with Dirichlet boundary conditions, i.e. 
$\Delta:\dom(\Delta)=H^2(\Omega)\cap H^1_0(\Omega)\rightarrow L^2(\Omega)$ and $H^1_0(\Omega)$ denotes the subset of functions that vanish at the boundary.
Then operator $-\Delta$ is self-adjoint on $\gotH$ and positive. For any $\alpha \in (0,1)$ the fractional power 
of operator $-\Delta$ is defined on the domain $\dom((-\Delta)^{\alpha})$, i.e. 
$(-\Delta)^{\alpha}: \dom((-\Delta)^{\alpha})\rightarrow L^2(\Omega)$. The domain is given by a fractional Sobolev space and for $\alpha>1/2$, we have
$\dom((-\Delta)^{\alpha}) = H_0^{2\alpha}(\Omega) \subset H^{2\alpha}(\Omega)$ (see \cite{LionsMagenes1972} for more information).

Moreover let $B(t)$ denote a time-dependent scalar-valued multiplication operator given by
\be\la{eq:4.2}
\begin{split}
 (B(t)f)(x) =&V(t,x)f(x),\\
 \dom(B(t))=& \{f\in L^2(\cI, \gotH): V(\cdot,x)f(x)\in L^2(\cI, \gotH)\} \ 
\end{split}
\ee
where $V:\cI\times\Omega\rightarrow \mathbb R$ is measurable.
We assume that the potential $V(\cdot,\cdot)$ is real and non-negative. 
Then $B(t)$ is obviously self-adjoint and non-negative on $\gotH$.
\begin{thm}\label{TheoremExample}
Let $A$ be the Laplacian operator $-\gD$ with Dirichlet boundary conditions in $L^2(\Omega)$, see above. 
Further, let $\{B(t)\}_{i\in\cI}$ be the family of multiplication operators defined by \eqref{eq:4.2}. If $V(\cdot,\cdot): \cI \times \Omega \longrightarrow \dR$ is measurable, real, non-negative with regularity $V\in L^\infty(\cI,L^{2+\varepsilon}(\Omega))\cap C^\beta(\cI,L^{1+\varepsilon}(\Omega))$ for $\gb \in (0,1)$ and some $\varepsilon>0$, then 
the assumptions (S1)-(S3) are satisfied with $\alpha\in[3/4, 1)$.  Moreover, if
$\gb > 2\ga -1$  then the converging rates of  Theorem \ref{thm:3.11}, Corollary \ref{cor:3.12} and Theorem \ref{thm:3.12} hold.
\end{thm}
\begin{proof}
Since $\Omega$ is bounded there one has $\inf\gs(A) > 0$ which does not satisfy $A \ge I$ in general and, hence, assumption (S1) is not satisfied. 
Nevertheless $\inf\gs(A) > 0$ is sufficient to prove the converging results. So we can believe that (S1) is satisfied. 

Let $\alpha \geq 3/4$. Using the Sobolev space embeddings, we get that 
$H^{2\alpha}(\Omega)\subset L^\gamma(\Omega)$ for any $\gamma\in[2,\infty[$. Hence, if $V\in L^\infty(\cI,L^{2+\varepsilon}(\Omega))$, 
we conclude that the function $[0,T] \ni t \mapsto B(t)(-\Delta)^{-\alpha} $ is essentially operator-norm bounded in $t \in \cI$ and thus, (S2) is satisfied. Now, we consider 
 \begin{align*}
 F(t):=(-\Delta)^{-\alpha}B(t)(-\Delta)^{-\alpha} : L^2(\Omega)\rightarrow
 H^{2\alpha}(\Omega)\subset L^2(\Omega).
\end{align*}
 The function $F(\cdot):\cI \rightarrow \cal L(\gotH)$ is bounded for fixed $t\in [0,T]$ if for any 
$f,g \in H^{2\alpha}(\Omega)$ the function $\langle f, B(t) g\rangle$ is bounded. This holds since 
$V(t, \cdot)\in L^{1+\varepsilon}(\Omega)$ and $H^{2\alpha}(\Omega)\subset L^\gamma(\Omega)$ for any 
$\gamma\in[2,\infty[$. Hence we conclude that (S3) is satisfied and the claim is proved.
\eproof
\end{proof}
Theorem \ref{TheoremExample} provides a convergence rate of an approximation of 
the solution of (\ref{EvolProbLaplaceAndTimeDependentPotential}) by the time-ordered product
 \begin{align}
  \wt V_n(t,s) = \prod_{j=1}^n e^{-\frac{t-s}{n}V(\frac{jt + (n-j)s}{n},\cdot)}e^{\frac{t-s}{n}\Delta}
 \end{align}
 This looks elaborate, but is indeed simple. There are strategies to compute the semigroup of the Laplace operator for bounded domains and there are also explicit formulas on special domains like disks etc. The factors $e^{-\tau V(t_j)}$, $j = 1,2, \ldots, n$ are scalar valued and can be easily computed.

\section*{Acknowledgment}

We thank Takashi Ichinose and Hideo Tamura for the explanation of details of the proof of Theorem 1.1 of \cite{IchinoseTamura1998}, 
which makes possible to prove Lemma~\ref{lem:5.2} and Lemma~\ref{lem:5.3}.

\vspace{10mm}

\def\cprime{$'$}

\end{document}